\documentclass{amsart}
\usepackage{amsmath,amssymb,enumerate,hyperref,mathrsfs,verbatim,xypic}
\usepackage{graphicx,color}

\newcommand{\Z}{{\mathbb{Z}}}
\newcommand{\N}{{\mathbb{N}}}

\newtheorem{thm}{Theorem} 

\newtheorem{prop}[thm]{Proposition}

\newtheorem*{corollary}{Corollary}

\begin{document}

\title{Symmetric Shannon capacity is the independence number minus $1$}


\author{Tam\'as Terpai}
\address{Lor\'and E\"otv\"os University, Department of Analysis\\1117 Budapest, P\'azm\'any P\'eter s\'et\'any 1/C\\Hungary}
\email{terpai@math.elte.hu}

\begin{abstract}
A symmetric variant of Shannon capacity is defined and computed.
\end{abstract}


\maketitle

\renewcommand{\thefootnote}{}
\footnotetext{{\it Acknowledgement:} The author is grateful to P\'eter Frenkel for coming up with the definition of symmetric Shannon capacity and to Endre Cs\'oka for organizing the Si\'ofok Workshop on Graphs and Graph Limits where the related questions were disseminated.}

\section{Introduction and motivation}
Given a finite graph $G$ with vertices $V=\{ v_1,\dots,v_n \}$ and a total weight $k\in \N$, we consider the following graph $G[k]$. The vertices of $G[k]$, forming the set $V[k]$, are configurations of $k$ identical pebbles in the vertices of $G$; that is, functions $f:V \to \Z_{\geq 0}$ such that $f(v_1)+\dots+f(v_n)=k$, which we consider also as equivalence classes of maps from a set of cardinality $k$ to $V$ with two maps being equivalent if they only differ in a permutation of the $k$ elements. We connect two vertices $f$ and $g$ of $G[k]$ if we can move some pebbles in the configuration $f$ to neighbours in $G$ to obtain the configuration $g$. The motivation for this construction is that if the pebbles were distinguishable, then it would yield the $k$-times strong direct power $G^{\boxtimes k}$ and the asymptotics of its independence number with respect to $k$ would define the Shannon capacity of $G$, which is a notoriously hard to compute graph parameter (e.g. computing it even for the $5$-cycle takes a lot of effort, see \cite{L}). In fact, the symmetric group $S_k$ acts on $G^{\boxtimes k}$ by permuting the coordinates, and $G[k] = G^{\boxtimes k}/S_k$.
\par
The graph $G[k]$ has ${k+n-1 \choose n-1}=O\left(k^{n-1}\right)$ vertices, therefore its independence number $\alpha(G[k])$ can also grow only polynomially. Set
$$
F(G) = \lim_{k\to \infty} \frac{\log \alpha(G[k])}{\log k}.
$$
A priori it is not clear whether such a limit exists: in contrast to the case of distinguishable pebbles, there is no obvious superadditivity property for $\log \alpha(G[k])$ that would ensure convergence.
It is however easy to see that $\alpha(G[k]) \geq {k+\alpha(G)-1 \choose \alpha(G)-1}$ -- the configurations supported on a fixed independent set $W \subset V$ are clearly independent in $G[k]$ -- and that $\alpha(G[k]) \leq {k+\theta(G)-1 \choose \theta(G)-1}$, where $\theta(G)$, the clique covering number of $G$, is the minimal number of cliques covering the vertices of $G$ -- for any clique covering $V=V_1 \cup \dots \cup V_\theta$ whenever the configurations $f$, $g \in V[k]$ have the same number of pebbles on the vertices of all $V_j$, they can be moved into one another by using only the edges within the $V_j$.

\section{C${}_5$ in particular}
The simplest graph for which $\theta(G)>\alpha(G)$ (and thus the estimates above do not determine $\alpha(G[k])$) is the $5$-cycle $C_5$. It can be checked that for $k \leq 9$ we have $\alpha(C_5[k])=k+1$, but for general $k$ we can only prove the following, weaker result:
\begin{prop}
For any natural number $k$ we have $\alpha(C_5[k]) \leq \left\lfloor\frac{5(k+2)(k+1)}{2(k+5)}\right\rfloor$.
\end{prop}

\begin{proof}
For simplicity let us index the vertices of $C_5$ by natural numbers from $1$ to $5$ in such a way that vertices $v_j$ and $v_{j+1}$ are connected by an edge -- labelled $e_j$ -- for $j=1,\dots,4$, and the vertices $v_1$ and $v_5$ are also connected by the edge $e_5$.

\begin{figure}[h]
\resizebox{0.4\columnwidth}{!}{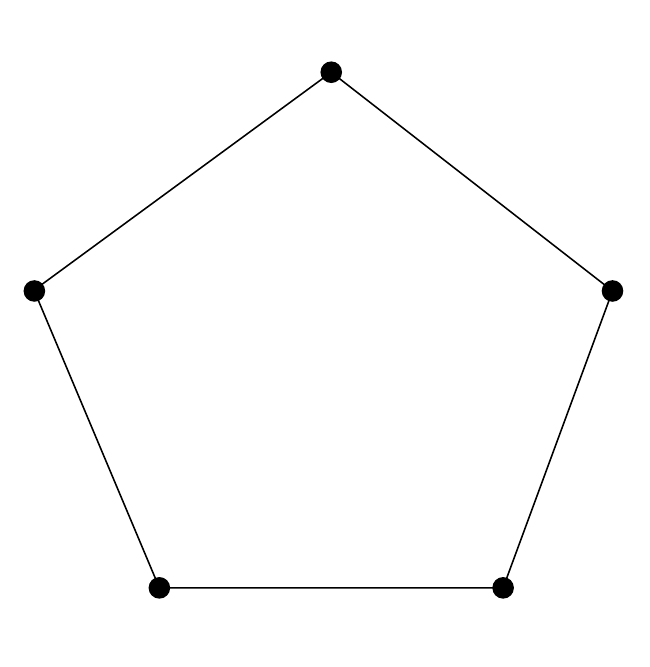}
\end{figure}

Consider {\it edge configurations} of weight $k$ on $C_5$: functions $\psi:E(C_5) \to \Z_{\geq 0}$ such that $\psi(e_1)+\dots+\psi(e_5)=k$, also considered as ways to distribute $k$ indistinguishable pebbles among the edges of $C_5$. We declare an edge configuration $\psi$ and a vertex configuration $f$ to be adjacent if one can obtain $\psi$ from $f$ by moving each pebble to an adjacent edge. Note that if $f$ and $g$ are vertex configurations, then they are neighbours in $C_5[k]$ exactly if there exists an edge configuration $\psi$ adjacent to both of them. As a consequence, if $A\subset V[k]$ is an independent set, then the sets of edge configurations adjacent to the elements of $A$ are disjoint; in particular, the sum of their sizes cannot exceed $k+4 \choose 4$, the total number of edge configurations of weight $k$. Unfortunately, the upper bound on $\vert A \vert$ obtained from this observation is insufficient as there are vertex configurations that are adjacent to few edge configurations: for example, only $k+1$ edge configurations are adjacent to the vertex configuration where all the pebbles are assigned to the same vertex. To circumvent this problem, we count only edge configurations that assign zero weight to at least one pair of non-adjacent edges. Let $S=S_1\cup S_2\cup S_3 \cup S_4 \cup S_5$ denote the set of such edge configurations, where $S_j$ is the set of those edge configurations that assign zero weight to edges $e_{j-1}$ and $e_{j+1}$ (with indices considered modulo $5$). Finally, as the sets $S_j$ are not disjoint, we shall consider each edge configuration weighted by the number of sets $S_j$ that contain it.

Now let $f \in V[k]$ be an arbitrary vertex configuration. Edge configurations in $S_1$ that are adjacent to $f$ have no pebbles on edges $e_2$ and $e_5$ and hence must be formed by moving all pebbles in $v_1$ and $v_2$ to the edge $e_1$, all pebbles in $v_3$ to the edge $e_3$ and all pebbles in $v_5$ to the edge $e_4$. The only freedom left is the distribution of the pebbles in $v_4$ to $e_3$ and $e_4$; the number of such distributions is $f(v_4)+1$. Repeating this argument for $S_2$ to $S_5$, we count $f(v_1)+1+\dots+f(v_5)+1=k+5$ edge configurations adjacent to $f$, each as many times as there are sets $S_j$ that contain it. The total weight of the edge configurations in $S$ is $\vert S_1 \vert + \dots + \vert S_5 \vert=5{k+2 \choose 2}$, hence there cannot be more than $\left\lfloor \frac{5{k+2 \choose 2}}{k+5} \right\rfloor = \left\lfloor \frac{5(k+2)(k+1)}{2(k+5)} \right\rfloor$ pairwise non-neighbouring vertex configurations, as claimed.
\end{proof}

\begin{corollary}
$F(C_5)=1$.
\end{corollary}

\par
\noindent{\bf Open question:} is $\alpha(C_5[k])=k+1$ for all $k$?
\par

\section{General case}

\begin{thm}
For any finite graph $G$ we have $\alpha(G[k])=O(k^{\alpha(G)-1})$.
\end{thm}

\begin{proof}
We prove our statement by induction on the cardinality of $V$. For $\vert V \vert =1$ the claim is trivial. Assume now that the statement holds for all graphs with less vertices than $\vert V \vert$, and let $H \subset V[k]$ be an arbitrary independent vertex set in $G[k]$. Partition $H$ into pieces $H_1$, $\dots$, $H_n$ such that in $H_j$ the vertex $v_j$ carries the greatest weight; in particular, this weight is at least $\frac{k}{n}$.
\par
Denote by $N(v_j) \subset V$ the union of the vertex $v_j$ and its neighbours in $G$. Each $H_j$ we further divide into chunks $H_{j; m, b_1, \dots, b_n}$, where the chunk $H_{j; m, b_1, \dots, b_n}$ contains those elements of $H_j$ in which the total weight of $v_j$ and its neighbours is $m$ and for all $i \leq n$ the weight of the vertex $v_i$ falls into the interval $\left[b_i\frac{k}{2n^2},(b_i+1)\frac{k}{2n^2}\right)$. The index $j$ can take $n$ different values, $m$ can take $k+1$ different values, and each $b_i$ can be assumed to be between $0$ and $2n^2$ (the rest are always empty), so altogether we get at most $n(k+1)(2n^2+1) = O(k)$ chunks (with an implied dependence on $n$). Pick now an arbitrary chunk $H_{j; m, b_1, \dots, b_n}$ and consider its elements as configurations on $G \setminus N(v_j)$ (with the weights on the elements of $N(v_j)$ omitted). These configurations all have weight $k-m$ and we claim that they form an independent set in $(G\setminus N(v_j))[k-m]$. Indeed, if any two configurations $f$ and $g$ in $H_{j; m, b_1, \dots, b_n}$ either coincide or are neighbours in $(G\setminus N(v_j))[k-m]$, then there exists a pebble transport from $f$ to $g$ on $G\setminus N(v_j)$. On the other hand, there exists a pebble transport from $f$ to $g$ on $N(v_j)$ as well: we send $f(v_i)-\left \lceil b_i\frac{k}{2n^2} \right\rceil$ pebbles from each vertex $v_i$ that is a neighbour of $v_j$ to $v_j$, and we send $g(v_i)-\left\lceil b_i\frac{k}{2n^2} \right\rceil$ pebbles from $v_j$ to each $v_i$ that is a neighbour of $v_j$. Combining these two pebble transports yields a pebble transport from $f$ to $g$.
\par
This means that the chunk $H_{j; m, b_1, \dots, b_n}$ has size at most $O(k^{\alpha(G\setminus N(v_j))-1})$ by the induction hypothesis. But any independent set in $G\setminus N(v_j)$ can be extended by $v_j$ to form an independent set in $G$, hence $\alpha(G\setminus N(v_j)) \leq \alpha(G)-1$. Summing this estimate for all the $O(k)$ chunks we get that $\vert H\vert = O(k^{\alpha(G)-1})$, proving the induction step.
\end{proof}

\begin{corollary}
For any finite graph $G$ we have $F(G)=\alpha(G)-1$.
\end{corollary}

\noindent{\bf Open question:} is $\alpha(G[k])={k+\alpha(G)-1 \choose \alpha(G)-1}$, or at least does
$$
\lim_{k \to \infty} \frac{\alpha(G[k])}{k^{\alpha(G)-1}}=\frac{1}{(\alpha(G)-1)!}
$$
hold?

\end{document}